\documentclass[12pt]{amsart}

\usepackage{fullpage}
\usepackage[utf8]{inputenc}
\usepackage[T1]{fontenc}
\usepackage{graphicx}
\usepackage{amsmath,amsfonts,amssymb,mathtools}
\usepackage{stmaryrd}
\usepackage{graphicx}
\usepackage{mathrsfs,dsfont}
\usepackage{amsmath,amsthm,amsthm}

\usepackage{hyperref}

\usepackage{enumerate}

\usepackage{color}

\newcommand{\E}{\mathbb{E}}
\newcommand{\PP}{\mathbb{P}}
\newcommand{\R}{\mathbb{R}}
\newcommand{\N}{\mathbb{N}}

\newtheorem{theo}{Theorem}[section]

\newtheorem{cor}[theo]{Corollary}
\newtheorem{rem}[theo]{Remark}

\newtheorem{propo}[theo]{Proposition}
\newtheorem{lemma}[theo]{Lemma}

\begin{document}

\title{Analysis of some splitting schemes for the stochastic Allen-Cahn equation}


\author{Charles-Edouard Br\'ehier}
\address{Univ Lyon, CNRS, Université Claude Bernard Lyon 1, UMR5208, Institut Camille Jordan, F-69622 Villeurbanne, France}
\email{brehier@math.univ-lyon1.fr}

\author{Ludovic Gouden\`ege}
\address{Université Paris-Saclay, CNRS - FR3487, Fédération de Mathématiques de CentraleSupélec, CentraleSupélec, 3 rue Joliot Curie, F-91190 Gif-sur-Yvette, France}
\email{ludovic.goudenege@math.cnrs.fr}


\date{}

\begin{abstract}
We introduce and analyze an explicit time discretization scheme for the one-dimensional stochastic Allen-Cahn, driven by space-time white noise. The scheme is based on a splitting strategy, and uses the exact solution for the nonlinear term contribution.

We first prove boundedness of moments of the numerical solution. We then prove strong convergence results: first, $L^2(\Omega)$-convergence of order almost $1/4$, localized on an event of arbitrarily large probability, then convergence in probability of order almost $1/4$.

The theoretical analysis is supported by numerical experiments, concerning strong and weak orders of convergence.
\end{abstract}

\maketitle

\section{Introduction}\label{sec:intro}

In this article, we define and study new numerical schemes for the time discretization of the following Stochastic Partial Differential Equation (SPDE),
\[
\frac{\partial u(t,\xi)}{\partial t}=\frac{\partial^2 u(t,\xi)}{\partial \xi^2}+u(t,\xi)-u(t,\xi)^3+\dot{W}(t,\xi)
\]
driven by Gaussian space-time white noise, with $\xi\in(0,1)$ a one-dimensional space variable -- and homogeneous Dirichlet boundary conditions.

The Allen-Cahn equation has been introduced \cite{AllenCahn:79} as a model for a two-phase system driven by the Ginzburg-Landau energy
\[
\mathcal{E}(u) = \int |\nabla u|^{2}+\frac{1}{\varepsilon^{2}} V(u),
\]
where $u$ is the ratio of the two species densities, and $V=(u^2-1)^2$ is a double well potential. The first term in the energy models the diffusion of the interface between the two pure phases, and the second one pushes the solution to two possible stable states $\pm 1$ (named the pure phases, i.e. minima of $V$). The behavior of the interface in the regime $\varepsilon\to 0$ is described in terms of mean curvature flow, see for instance~\cite{ChenHilhorstLogak:97, ChenGigaGoto:91, EvansSonerSouganidis:92, EvansSpruck:91, EvansSpruck:92}. 

The stochastic version of the Allen-Cahn equation models the effect of thermal perturbations by a additional noise term, see for instance~\cite{FarisJona-Lasinio:82, Funaki:95, KarsoulakisKossiorisLakkis:07}. The behavior as $\varepsilon \rightarrow 0$ has been studied for instance in~\cite{Funaki:95} in dimension 1 (with space-time white noise), and \cite{Funaki:99, Weber:10} in higher dimension (with more regular noise).

The stochastic Allen-Cahn equation is also a popular model for the study and simulation of rare events in infinite dimensional stochastic system, see for instance~\cite{BrehierGazeauGoudenegeRousset:15, KohnOttoReznikoffVanden-Eijnden:07, RollandBouchetSimonnet:16, Vanden-EijndenWeare:12}.

\bigskip

In this work, our aim is to study numerical schemes for the stochastic Allen-Cahn equation. In the theoretical analysis, we only focus on the temporal discretization. To perform numerical simulations, a spatial discretization is required: we use a standard finite difference method.

Numerical schemes for SPDEs have been extensively studied in the last two decades, see for instance the monographs~\cite{JentzenKloeden, Kruse, LordPowellShardlow}. Compared with the numerical discretization of Stochastic Differential Equations (SDEs), both temporal and spatial discretization are required. In addition, the temporal regularity of the solutions of SPDEs depends heavily on the spatial regularity of the noise perturbation, and this affects orders of convergence. For instance, consider equations with globally Lipschitz continuous coefficients. For SDEs, the solutions are H\"older continuous with exponents $\alpha<1/2$, and the Euler-Maruyama scheme has in general a strong order of convergence $1/2$ and a weak ordre $1$: see for instance the monographs~\cite{KloedenPlaten:92, Milstein:95}. For SPDEs, driven by space-time white noise, the solutions are only H\"older continuous with exponents $\alpha<1/4$, and (explicit or implicit) Euler type schemes only have strong order $1/4$ and weak order $1/2$. We recall that strong convergence usually refers to convergence in mean-square sense: see for instance
~\cite{DavieGaines:01, Gyongy:98, Gyongy:99, GyongyMillet:05, GyongyMillet:07, GyongyNualart:95, GyongyNualart:97, Jentzen:11, JentzenKloedenWinkel:11, JentzenKloeden:09, KloedenShott:01, LordTambue:13, Printems:01, Wang:17}.
 Weak convergence refers to convergence in distribution: recent contributions are
~\cite{BrehierDebussche:17, ConusJentzenKurniawan:14, Debussche:11, JentzenKurniawan:15}.

For discretization of equations such as the stochastic Allen-Cahn equation, the main difficulty is the polynomial coefficient, which is not globally Lipschitz continuous. Standard schemes with explicit discretization of such coefficients cannot be applied. Defining efficient numerical schemes for stochastic equations with non globally Lipschitz continuous coefficients is delicate: see for instance the recent works~\cite{HutzenthalerJentzenKloeden:12}, and the monograph~\cite{HutzenthalerJentzen:15}, and references therein. The general methodology has been recently applied to various examples of SPDEs: see for instance~\cite{BeckerJentzen:16, HutzenthalerJentzenSalimova:16, JentzenPusnik:15, JentzenPusnik:16, JentzenSalimovaWelti:17}. We also mention~\cite{GyongySabanisSiska:16} for the analysis of a tamed Euler scheme for a class of SPDEs.

The case of the Allen-Cahn equation has been treated in the recent work~\cite{BeckerGessJentzenKloeden:17}, with a scheme based on an exponential integrator and a tamed discretization of the nonlinear coefficient. References~\cite{KovacsLarssonLindgren:15_1, KovacsLarssonLindgren:15_2} present analysis of implicit schemes. Finally, a Wong-Zakai approximation has been considered in~\cite{LiuQiao:17}.

\bigskip

In this work, we introduce new schemes, based on a splitting (also referred to as splitting-up) strategy, see for instance~\cite{BessaihBrzezniakMillet:14, BessaihMillet:18, GyongyKrylov:03_2,GyongyKrylov:03} in the SPDE case, and~\cite{JeongLeeLeeShinKim:16, LeeLee:15} for the deterministic Allen-Cahn equation. Indeed, the solution of the ordinary differential equation $\dot{z}=z-z^3$ has a known explicit expression. The splitting strategy then consists in solving separately the contributions of the linear coefficient with the noise, and of the nonlinear coefficient. Several schemes may be chosen to treat the first contribution: exponential and linear implicit Euler integrators may be used.

We mostly focus on the following scheme (other schemes will be defined when performing numerical simulations). Let us write the stochastic Allen-Cahn equation as an evolution equation in the sense of Da Prato-Zabczyk,~\cite{DaPratoZabczyk:14}:
\[
dX(t)=AX(t)dt+\Psi_0(X(t))dt+dW(t),
\]
with $\Psi_0(x)=x-x^3$. Then the numerical scheme, with time-step size $\Delta t>0$, is defined by the recursion:
\begin{equation*}
\begin{cases}
Y_{n}=\Phi_{\Delta t}(X_n),\\
X_{n+1}=S_{\Delta t}X_n+S_{\Delta t}\bigl(W((n+1)\Delta t)-W(n\Delta t)\bigr),
\end{cases}
\end{equation*}
where $S_{\Delta t}=(I-\Delta t A)^{-1}$ corresponds to the choice of a linear implicit Euler integrator, and
\[
\Phi_{\Delta t}(z)=\frac{z}{\sqrt{z^2+(1-z^2)e^{-2\Delta t}}}.
\]
Observe that the discrete-time process $\bigl(X_n\bigr)_{n\in\N}$ may be interpreted as the solution of a standard linear implicit Euler scheme for a modified SPDE, with nonlinear coefficient $\Psi_{\Delta t}(z)=\Delta t^{-1}(\Phi_{\Delta t}(z)-z)$, in the spirit of~\cite{HighamMaoStuart:02}. The coefficient $\Psi_{\Delta t}$ satisfies the same type of one-sided Lipschitz condition as $\Psi_0$, uniformly with respect to $\Delta t$.

\bigskip

Our first contribution is the analysis of the splitting scheme introduced above. We first prove moment bounds, uniform with respect to $\Delta t$. Our main result, Theorem~\ref{theo:main}, is a strong convergence result, with order of convergence almost $1/4$, localized on an event of arbitrarily large probability, in the spirit of~\cite{BessaihBrzezniakMillet:14}. We also state and prove several straightforward consequences of Theorem~\ref{theo:main}, related with other types of convergence:
\begin{itemize}
\item convergence in mean-square sense, with no order of convergence,
\item convergence in probability of order almost $1/4$, in the spirit of~\cite{Printems:01},
\item weak convergence, with order almost $1/4$, rejecting exploding trajectories in the spirit of~\cite{MilsteinTretyakov:05}.
\end{itemize} 
We also provide numerical simulations to illustrate the rates of convergence of the scheme introduced above, and compare with a few variants.

These numerical experiments lead us to conjecture that our results may be improved as follows. First, we conjecture that the strong order is equal to $1/4$, in the standard sense, {\it i.e.} that it is possible to get rid of the localization in Theorem~\ref{theo:main}. But we expect the analysis to be considerably more complex and similar to~\cite{BeckerGessJentzenKloeden:17}. Second, we conjecture that the weak order is equal to $1/2$, when considering sufficiently smooth test functions. Again the analysis requires more complex arguments. We plan to investigate these questions in future works.

\bigskip

This article is organized as follows. The setting is introduced in Section~\ref{sec:setting}. The splitting schemes are introduced in section \ref{sec:num-sche}. Results concerning the auxiliary flow map $\Phi_{\Delta t}$ are given in Section~\ref{sec:aux_flow}. A priori bounds on the moments of the numerical solutions are given in Section~\ref{sec:moments}. Our main results are stated in Section~\ref{sec:convergence_statements}, and their proofs are given in Sections~\ref{sec:convergence_proof1} and~\ref{sec:convergence_proof2}. Numerical experiments to investigate strong and weak orders of convergence are reported in Section~\ref{sec:num-expe}.

\section{Setting}\label{sec:setting}

We work in the standard framework of stochastic evolution equations with values in infinite dimensional separable Hilbert and Banach spaces. We refer for instance to~\cite{Cerrai:01, DaPratoZabczyk:14} for details. Let $H=L^2(0,1)$, and $E=\mathcal{C}([0,1])$. We use the following notation: for $x_1,x_2\in H$, $x\in E$,
\[
\langle x_1,x_2\rangle =\int_{0}^{1}x_1(\xi)x_2(\xi)d\xi~,\quad \|x_1\|_H=\bigl(\int_0^1 x_1(\xi)^2d\xi\bigr)^{\frac{1}{2}}~,\quad |x|_E=\underset{\xi\in[0,1]}\max |x(\xi)|.
\]
To simplify notation, we often write $\|x\|=\|x\|_H$ and $|x|=|x|_E$.

\subsection{Assumptions}

\subsubsection{Linear operator}

Let $A$ denote the unbounded linear operator on $H$, with
\[
\begin{cases}
D(A)=H^2(0,1)\cap H_0^1(0,1)\\
Ax=\partial_\xi^2x,~\forall~x\in D(A).
\end{cases}
\]
It is well-known that $A$ generates a strongly continuous semi-group, both on $H$ and on $E$. We use the notation $\bigl(e^{tA}\bigr)_{t\ge 0}$. More precisely, it is an analytic semi-group.

Finally, let $e_n=\sqrt{2}\sin(n\pi\cdot)$ and $\lambda_n=n^2\pi^2$, for $n\in\N$. Note that $Ae_n=-\lambda_n e_n$, and that $\bigl(e_n\bigr)_{n\in\N}$ is a complete orthonormal system of $H$.

\subsubsection{Wiener process}

Let $\bigl(\Omega,\mathcal{F},\PP\bigr)$ denote a probability space, and consider a family $\bigl(\beta_n\bigr)_{n\in\N}$ of independent standard real-valued Wiener processes. Then set
\[
W(t)=\sum_{n\in\N}\beta_n(t)e_n.
\]
This series does not converge in $H$. However, if $\tilde{H}$ is an Hilbert space, and $L\in\mathcal{L}_2(H,\tilde{H})$ is a linear, Hilbert-Schmidt, operator, then $LW(t)$ is a Wiener process on $\tilde{H}$, centered and with covariance operator $LL^{\star}$.

\subsubsection{Stochastic convolution}

The linear equation, with additive noise,
\[
dX(t)=AX(t)dt+dW(t),~Z(0)=0,
\]
admits a unique (global) mild solution in $H$
\[
X(t)=\int_{0}^{t}e^{(t-s)A}dW(s)=W^A(t),
\]
called the stochastic convolution.

Moreover, this process is continuous with values in $E$. Moment estimates are satisfied: for all $T\in(0,\infty)$ and all $p\in\N$, there exists $C_{p}(T)\in(0,\infty)$ such that
\begin{equation}\label{eq:moment_stochastic-convolution}
\E\bigl[\sup_{0\le t\le T}\|W^A(t)\|_H^{2p}+\sup_{0\le t\le T}|W^A(t)|_E^{2p}\bigr]\le C_{p}(T).
\end{equation}

\subsection{Allen-Cahn equation}

The potential energy function $V:\R\to\R$ is defined by
\[
V(z)=\frac{z^4}{4}-\frac{z^2}{2}.
\]
Then the function $\Psi_0=-V'$ satisfies a one-sided Lipschitz condition: for all $z_1,z_2\in\R$,
\[
\bigl(\Psi_0(z_2)-\Psi_0(z_1)\bigr)\bigl(z_2-z_1\bigr)\le |z_2-z_1|^2.
\]
However, $\Psi_0$ is not globally Lipschitz continuous.

In this article, we consider the stochastic Allen-Cahn equation, with additive space-time white noise, {\it i.e.} the Stochastic Partial Differential Equation
\begin{equation}\label{eq:AC}
dX(t)=AX(t)dt+\Psi_0(X(t))dt+dW(t), X(0)=x_0,
\end{equation}
with an initial condition $x_0\in E$.

We quote the following well-posedness result, see for instance~\cite[Chapter~6]{Cerrai:01}.
\begin{propo}\label{propo:AC}
Let $T\in(0,\infty)$. There exists a unique global mild solution $\bigl(X(t)\bigr)_{0\le t\le T}$ of~\eqref{eq:AC}, with values in $E$. Moreover, for every $p\in\N$, there exists $C_p(T)\in(0,\infty)$ such that
\[
\E\bigl[\underset{0\le t\le T}\sup|X(t)|_{E}^{2p}\bigr]\le C_p(T)\bigl(1+|x_0|_{E}^{2p}\bigr).
\]
\end{propo}

\section{Numerical schemes}\label{sec:num-sche}

Since the coefficient $\Psi_0$ is not globally Lipschitz continuous, it is well-known that explicit discretization schemes are not appropriate -- unless combined with a taming strategy, as in~\cite{BeckerGessJentzenKloeden:17, GyongySabanisSiska:16} for instance. Fully implicit schemes are expensive, and split-step schemes such as defined in~\cite{KovacsLarssonLindgren:15_1, KovacsLarssonLindgren:15_2}, with an implicit discretization for the contributions depending on $\Psi_0$, may be defined.

In the case of Allen-Cahn equations, our strategy, detailed below, consists in replacing these implicit steps with the exact solution of the flow associated with $\Psi_0$, in the spirit of~\cite{JeongLeeLeeShinKim:16, LeeLee:15}.

\subsection{Splitting schemes}


Introduce the auxiliary ordinary differential equation
\[
\dot{z}=\Psi_0(z),~z(0)=z_0\in \mathbb{R}.
\]
The flow of this equation is known: the unique solution $\bigl(z(t)\bigr)_{t\ge 0}$ is given by
\begin{equation}\label{eq:defPhi}
z(t)=\Phi_t(z_0)=\frac{z}{\sqrt{z_0^2+(1-z_0^2)e^{-2t}}}~,\quad t\ge 0.
\end{equation}

The splitting schemes we consider may be written in the following abstract form: let $\Delta t>0$ denote the time-step size, then
\begin{equation}\label{eq:scheme_Gamma}
\begin{cases}
Y_{n}=\Phi_{\Delta t}(X_n),\\
X_{n+1}=\Gamma\bigl(Y_n,\Delta t,(W(t))_{n\Delta t\le t\le (n+1)\Delta t}\bigr).
\end{cases}
\end{equation}
To complete the definition of the numerical schemes, it remains to provide the definition of the mapping $\Gamma$, corresponding to the approximation of the stochastic convolution. In the analysis below, three examples are considered:
\[
\begin{cases}
\Gamma^{{\rm exact}}\bigl(y,\Delta t,(W(t))_{n\Delta t\le t\le (n+1)\Delta t}\bigr)=e^{\Delta t A}y+\int_{n\Delta t}^{(n+1)\Delta t}e^{((n+1)\Delta t-t)A}dW(t),\\
\Gamma^{{\rm expo}}\bigl(y,\Delta t,(W(t))_{n\Delta t\le t\le (n+1)\Delta t}\bigr)=S_{\Delta t}^{{\rm expo}}y+S_{\Delta t}^{{\rm expo}}\Delta W_n,\\
\Gamma^{{\rm imp}}\bigl(y,\Delta t,(W(t))_{n\Delta t\le t\le (n+1)\Delta t}\bigr)=S_{\Delta t}^{{\rm imp}}y+S_{\Delta t}^{{\rm imp}}\Delta W_n
\end{cases}
\]
where $\Delta W_n=W\bigl((n+1)\Delta t\bigr)-W\bigl(n\Delta t\bigr)$ are Wiener increments, and with linear operators
\[
S_{\Delta t}^{{\rm expo}}=e^{\Delta tA} \quad,\quad S_{\Delta t}^{{\rm imp}}=\bigl(I-\Delta t A\bigr)^{-1}.
\]

Numerical experiments, see Section~\ref{sec:num-expe}; will also be performed for other schemes, based on different splitting strategies.

When using the splitting scheme with $\Gamma=\Gamma^{{\rm exact}}$, both sub-steps are solved exactly. On the contrary, when using the other examples, there is an error due to the discretization of the stochastic convolution.

We use the notation $X_n^{{\rm exact}}$, $X_n^{{\rm expo}}$ and $X_n^{{\rm imp}}$, when choosing $\Gamma=\Gamma^{{\rm exact}}$, $\Gamma^{{\rm expo}}$ and $\Gamma=\Gamma^{{\rm imp}}$ respectively. To simplify, we do not mention the dependence with respect to $\Delta t$.

\subsection{Auxiliary SPDE}

Define auxiliary functions $\Psi_t$, for $t>0$, as follows: for all $z\in\mathbb{R}$,
\begin{equation}\label{eq:defPsi}
\Psi_{t}(z)=\frac{\Phi_t(z)-z}{t}.
\end{equation}

An important tool in the analysis is the auxiliary equation
\begin{equation}\label{eq:aux}
dX^{(\Delta t)}(t)=AX^{(\Delta t)}(t)dt+\Psi_{\Delta t}\bigl(X^{(\Delta t)}(t)\bigr)dt+dW(t)~,\quad X^{(\Delta t)}(0)=x_0,
\end{equation}
with nonlinear coefficient $\Psi_0$ in~\eqref{eq:AC} replaced with $\Psi_{\Delta t}$.

Observe that the numerical schemes defined by~\eqref{eq:scheme_Gamma}, based on the splitting method, can be interpreted as standard numerical schemes for the auxiliary equation~\eqref{eq:aux}:
\[
\begin{cases}
X_{n+1}^{{\rm exact}}=e^{\Delta t A}X_n^{{\rm exact}}+\Delta te^{\Delta t A}\Psi_{\Delta t}(X_n^{{\rm exact}})+\int_{n\Delta t}^{(n+1)\Delta t}e^{((n+1)\Delta t-t)A}dW(t),\\
X_{n+1}^{{\rm expo}}=S_{\Delta t}^{{\rm expo}}X_n^{{\rm expo}}+\Delta tS_{\Delta t}^{{\rm expo}}\Psi_{\Delta t}(X_n^{{\rm expo}})+S_{\Delta t}^{{\rm expo}}\Delta W_n,\\
X_{n+1}^{{\rm imp}}=S_{\Delta t}^{{\rm imp}}X_n^{{\rm imp}}+\Delta tS_{\Delta t}^{{\rm imp}}\Psi_{\Delta t}(X_n^{{\rm imp}})+S_{\Delta t}^{{\rm imp}}\Delta W_n.
\end{cases}
\]
The schemes $X^{{\rm exact}}$ and $X^{{\rm expo}}$ correspond to versions of the exponential Euler scheme, applied to the auxiliary equation~\eqref{eq:aux}. The scheme $X^{{\rm imp}}$ is the standard linear implicit Euler scheme, applied to the auxiliary equation~\eqref{eq:aux}.

The three schemes are well-defined, for any value of $\Delta t>0$. Indeed, the mapping $\Psi_{\Delta t}$ is globally Lipschitz continuous.

\begin{rem}
One may also introduce the following scheme:
\[
X_{n+1}^{{\rm acc}}=e^{\Delta t A}X_n^{{\rm acc}}+(-A)^{-1}\bigl(I-e^{\Delta tA}\bigr)\Psi_{\Delta t}(X_n^{{\rm acc}})+\int_{n\Delta t}^{(n+1)\Delta t}e^{((n+1)\Delta t-t)A}dW(t).
\]
This scheme corresponds to the application of the accelerated exponential Euler scheme to the auxiliary equation. However, this scheme is not based on a splitting method.
\end{rem}

\subsection{Results concerning the auxiliary coefficients $\Phi_{\Delta t}$ and $\Psi_{\Delta t}$}\label{sec:aux_flow}

In this section, we state several results concerning the real valued mappings $\Phi_{\Delta t}$ and $\Psi_{\Delta t}$ defined by~\eqref{eq:defPhi} and~\eqref{eq:defPsi}, with $t=\Delta t$. Proofs are postponed to the Appendix~\ref{appendix}.

Note that the estimates below are uniform for $\Delta t\in[0,\Delta t_0]$, for any arbitrary $\Delta t_0>0$ -- and without loss of generality assume $\Delta t_0<1$. Moreover, the estimates are consistent when $\Delta t=0$, with $\Phi_0(z)=z$ and $\Psi_0(z)=z-z^3$.

The first result yields global Lipschitz continuity of $\Phi_{\Delta t}$.
\begin{lemma}\label{lem:Lip_Phi}
For every $\Delta t_0\in(0,1)$, for all $\Delta t\in[0,\Delta t_0]$, the mapping $\Phi_{\Delta t}$ is globally Lipschitz continuous, and the Lipschitz constant is bounded from above uniformly for $\Delta t\in[0,\Delta t_0]$. More precisely, for all $z_1,z_2\in\R$,
\[
\big|\Phi_{\Delta t}(z_2)-\Phi_{\Delta t}(z_1)\big|\le e^{\Delta t}|z_2-z_1|.
\]
\end{lemma}

The second result yields a one-sided Lipschitz condition for $\Psi_{\Delta t}$.
\begin{lemma}\label{lem:one-sided}
For every $\Delta t_0\in(0,1)$, for all $\Delta t\in[0,\Delta t_0]$, the mapping $\Psi_{\Delta t}$ satisfies a one-sided Lipschitz condition, uniformly for $\Delta t\in[0,\Delta t_0]$.  More precisely, for all $z_1,z_2\in\R$,
\[
\bigl(\Psi_{\Delta t}(z_2)-\Psi_{\Delta t}(z_1)\bigr)\bigl(z_2-z_1\bigr)\le e^{\Delta t}(z_2-z_1)^2.
\]
\end{lemma}

In addition to the one-sided Lipschitz condition from Lemma~\ref{lem:one-sided} above, $\Psi_{\Delta t}$ is locally Lipschitz continuous.
\begin{lemma}\label{lem:Lip_Psi}
For every $\Delta t_0\in(0,1)$, there exists $C(\Delta t_0)\in(0,\infty)$, such that for all $\Delta t\in[0,\Delta t_0]$ and all $z_1,z_2\in\R$,
\[
\big|\Psi_{\Delta t}(z_2)-\Psi_{\Delta t}(z_1)\big|\le C(\Delta t_0)|z_2-z_1|\bigl(1+|z_1|^3+|z_2|^3\bigr).
\]
In addition, for all $z\in\R$,
\[
|\Psi_{\Delta t}(z)|\le C(\Delta t_0)(1+|z|^4).
\]
\end{lemma}

Finally, the following result makes precise the speed of convergence of $\Psi_{\Delta t}$ to $\Psi_0$, when $\Delta t\to 0$.
\begin{lemma}\label{lem:errorPsi}
For every $\Delta t_0\in(0,1)$, there exists $C(\Delta t_0)\in(0,\infty)$ such that, for all $\Delta t\in[0,\Delta t_0]$ and $z\in \R$,
\[
|\Psi_{\Delta t}(z)-\Psi_0(z)|\le C(\Delta t_0)\Delta t(1+|z|^{5}).
\]
\end{lemma}

\subsection{Moment bounds}\label{sec:moments}

\subsubsection{Moment bounds for solutions of the auxiliary SPDE~\eqref{eq:aux}}

Using the one-sided Lipschitz condition for $\Psi_{\Delta t}$, from Lemma~\ref{lem:one-sided}, the same arguments used to get Proposition~\ref{propo:AC} yield the following moment bounds for the solution $X^{(\Delta t)}$ of the auxiliary equation, uniformly in $\Delta t\in(0,\Delta t_0]$, for arbitrary $\Delta t_0\in(0,1)$.

\begin{propo}\label{propo:aux}
Let $T\in(0,\infty)$. There exists a unique global mild solution $\bigl(X^{(\Delta t)}(t)\bigr)_{0\le t\le T}$ of~\eqref{eq:aux}, with values in $E$. Moreover, for every $p\in\N$, and every $\Delta t_0\in(0,1)$, there exists $C_p(T,\Delta t_0)\in(0,\infty)$ such that
\[
\underset{\Delta t\in(0,\Delta t_0]}\sup\E\bigl[\underset{0\le t\le T}\sup|X^{(\Delta t_0)}(t)|_{E}^{2p}\bigr]\le C_p(T,\Delta t_0)\bigl(1+|x_0|_{E}^{2p}\bigr).
\]
\end{propo}

Observe that, when $\Delta t>0$, the mapping $\Psi_{\Delta t}$ is globally Lipschitz continuous, the existence of moments is thus a standard result. The one-sided Lipschitz condition ensures that the estimate is uniform for $\Delta t\in(0,\Delta t_0]$.

\subsubsection{Moment bounds for solutions of the numerical schemes~\eqref{eq:scheme_Gamma}}

Let $\Delta t_0\in(0,1)$, and $\Delta t\in(0,\Delta t_0]$ denote a time-step size. Let $\bigl(X_n\bigr)_{n\in\N_0}$ be defined by the numerical scheme~\eqref{eq:scheme_Gamma}, with the mapping $\Gamma\in\left\{\Gamma^{{\rm exact}},\Gamma^{{\rm expo}},\Gamma^{{\rm imp}}\right\}$. For $T\in(0,\infty)$, let $N_{T,\Delta t}=\lfloor\frac{T}{\Delta t}\rfloor$.

\begin{propo}\label{propo:moment}
Let $T\in(0,\infty)$ and $\Delta t_0\in(0,1)$. For any $p\in\N$, there exists $C_p(t,\Delta t_0)\in(0,\infty)$ such that for all $\Delta t\in(0,\Delta t_0]$ and all $x_0\in E$,
\[
\E\bigl[\underset{0\le n\le N_{T,\Delta t}}\sup |X_n|_{E}^{2p}\bigr]\le C_p(T,\Delta t_0)(1+|x_0|_E^{2p}).
\]
\end{propo}

The proof uses the following result, in the situation with $\Psi_{0}$ replaced with $0$ in~\eqref{eq:AC}.
\begin{lemma}\label{lem:moment_omega}
Let $\bigl(\omega_n\bigr)_{n=0,\ldots,N_{T,\Delta t}}$ be defined by $\omega_0=0$, and
\[
\omega_{n+1}=\Gamma\bigl(\omega_n,\Delta t,(W(t))_{n\Delta t\le t\le (n+1)\Delta t}\bigr)
\]
Let $T\in(0,\infty)$ and $\Delta t_0\in(0,1)$. For any $p\in\N$, there exists $C_p(T,\Delta t_0)\in(0,\infty)$ such that for all $\Delta t\in(0,\Delta t_0]$,
\[
\E\bigl[\underset{0\le n\le N_{T,\Delta t}}\sup |\omega_n|_{E}^{2p}\bigr]\le C_p(T,\Delta t_0).
\]
\end{lemma}

Before sketching the proof of Lemma~\ref{lem:moment_omega}, we provide a detailed proof of Proposition~\ref{propo:moment}.
\begin{proof}[Proof of Proposition~\ref{propo:moment}]
Let $r_n=X_n-\omega_n$, for $n\in\left\{0,\ldots,N_{T,\Delta t}\right\}$.

Then $r_0=x_0$, and (with $S_{\Delta t}=S_{\Delta t}^{{\rm exact}}=e^{\Delta tA}$ when $\Gamma=\Gamma^{{\rm exact}}$)
\begin{align*}
r_{n+1}&=X_{n+1}-\omega_{n+1}=S_{\Delta t}\bigl(\Phi_{\Delta t}(X_n)-\omega_n\bigr)\\
&=S_{\Delta t}\bigl(\Phi_{\Delta t}(r_n+\omega_n)-\Phi_{\Delta t}(\omega_n)\bigr)+S_{\Delta t}\bigl(\Phi_{\Delta t}(\omega_n)-\omega_n\bigr).
\end{align*}

In addition, the linear operator $S_{\Delta t}\in\left\{S_{\Delta t}^{{\rm exact}},S_{\Delta t}^{{\rm expo}},S_{\Delta t}^{{\rm imp}}\right\}$ satisfies $|S_{\Delta t}x|_E\le |x|_E$, for all $x\in E$ and all $\Delta t>0$, as a consequence of the maximum principle for the Laplace operator. On the one-hand, using Lipschitz continuity of $\Phi_{\Delta t}$, see Lemma~\ref{lem:Lip_Phi}, then
\[
\big|S_{\Delta t}\bigl(\Phi_{\Delta t}(r_n+\omega_n)-\Phi_{\Delta t}(\omega_n)\bigr)\big|_E\le e^{\Delta t}|r_n|_E.
\]
On the other hand, thanks to Lemma~\ref{lem:Lip_Psi}, and the identity $\Phi_{\Delta t}(z)-z=\Delta t\Psi_{\Delta t}(z)$,
\[
\big|S_{\Delta t}\bigl(\Phi_{\Delta t}(\omega_n)-\omega_n\bigr)\big|_E\le C(\Delta t_0)\Delta t(1+|\omega_n|_E^4).
\]
The last two estimates prove that
\[
|r_{n+1}|_E\le e^{\Delta t}|r_n|_E+C(\Delta t_0)\Delta t(1+|\omega_n|_E^4),
\]
and by discrete Gronwall's Lemma, for all $n\in\left\{0,\ldots,N_{T,\Delta t}\right\}$,
\[
|r_n|_E\le C(T,\Delta t_0)\bigl(1+|x_0|_E+\underset{0\le m\le N_{T,\Delta t}}\sup|\omega_m|_E^4\bigr).
\]
Applying the estimate of Lemma~\ref{lem:moment_omega} then concludes the proof of Lemma~\ref{propo:moment}.
\end{proof}

It remains to give a sketch of proof of Lemma~\ref{lem:moment_omega}.
\begin{proof}[Sketch of proof of Lemma~\ref{lem:moment_omega}]
When $\Gamma=\Gamma^{{\rm exact}}$, then $\omega_n=W^A(n\Delta t)$, and thus the result is a straightforward consequence of~\eqref{eq:moment_stochastic-convolution}.

When $\Gamma=\Gamma^{{\rm expo}}$, define
\[
\tilde{\omega}(t)=e^{(t-n\Delta t)A}\omega_n+\int_{n\Delta t}^{t}e^{\Delta t A}dW(s)~,\quad n\Delta t\le t\le (n+1)\Delta t.
\]
When $\Gamma=\Gamma^{{\rm imp}}$, define
\[
\tilde{\omega}(t)=\omega_n+(t-n\Delta t)AS_{\Delta t}^{{\rm imp}}\omega_n+\int_{n\Delta t}^{t}S_{\Delta t}^{{\rm imp}}dW(s)~,\quad n\Delta t\le t\le (n+1)\Delta t.
\]
Note that $\tilde{\omega}$ is a continuous process, with values in $E$, and satisfies $\tilde{\omega}(n\Delta t)=\omega_n$. For all $t\ge 0$ and $\xi\in[0,1]$, let $\tilde{\omega}(t,\xi)=\tilde{\omega}(t)(\xi)$.

We claim that, for all $T\in(0,\infty)$, $p\in\N$, $\Delta t_0\in(0,1)$, and $\alpha\in(0,\frac{1}{4})$, there exists $C_{\alpha,p}(T,\Delta t_0)$ such that for all $\Delta t\in(0,\Delta t_0]$,
\begin{align*}
\E\big|\tilde{\omega}(t,\xi_2)-\tilde{\omega}(t,\xi_1)\big|^{2p}&\le C_{\alpha,p}(T,\Delta t_0)|\xi_2-\xi_1|^{4\alpha p}~,\quad \forall~t\in[0,T],~\xi_1,\xi_2\in[0,1],\\
\E\big|\tilde{\omega}(t_2,\xi)-\tilde{\omega}(t_1,\xi)\big|^{2p}&\le C_{\alpha,p}(T,\Delta t_0)|t_2-t_1|^{2\alpha p}~,\quad \forall~t_1,t_2\in[0,T],~\xi\in[0,1].
\end{align*}
The proof of this statement uses only standard arguments (see~\cite{DaPratoZabczyk:14}), it is left to the reader. Note that $\tilde{\omega}(0)=0$, and $\tilde{\omega}(t,0)=\tilde{\omega}(1)=0$ for all $t\ge 0$. Then the application of the Kolmogorov regularity criterion concludes the proof.
\end{proof}

\section{Convergence analysis of the splitting schemes}\label{sec:convergence}

Our main result, Theorem~\ref{theo:main}, and several consequences, are stated in Section~\ref{sec:convergence_statements}. Section~\ref{sec:convergence_proof1} is devoted to a detailed proof of Theorem~\ref{theo:main}, and the other results are then proved in Section~\ref{sec:convergence_proof2}.

\subsection{Statements}\label{sec:convergence_statements}


Our main result is the following.
\begin{theo}\label{theo:main}
Let $T\in(0,\infty)$, $\Delta t_0\in(0,1)$, and $\alpha\in(0,\frac{1}{4})$. There exists $C_{\alpha}(T,\Delta t_0)\in(0,\infty)$, such that, for every $\Delta t\in(0,\Delta t_0],~M\in\N$  and $x_0\in E$, with $\|(-A)^\alpha x_0\|_H<\infty$, then
\[
\E\bigl[\mathds{1}_{\Omega_{M}^{(\Delta t)}(T)}\sup_{0\le n\le N_{T,\Delta t}}\|X_n-X(n\Delta t)\|_H^2\bigr]\le K_\alpha(M,T,\Delta t_0)\Delta t^{2\alpha}(1+|x_0|_E^8+\|(-A)^\alpha x_0\|_H^2)
\]
with $K_\alpha(M,T,\Delta t_0)\le C_{\alpha}(T,\Delta t_0)(1+M^6)\exp(C_{\alpha}(T,\Delta t_0)M^6)$, and
\[
\Omega_{M}^{(\Delta t)}(T)=\left\{\underset{0\le k\Delta t\le T}\sup|X_k|_E+\underset{0\le t\le T}\sup|X^{(\Delta t)}(t)|_E\le M\right\}.
\]
Moreover, there exists $C(T,\Delta t_0)$ such that for every $x_0\in E$, $\Delta t\in(0,\Delta t_0]$ and every $M\in\N$
\[
\PP\bigl(\Omega_{M}^{(\Delta t)}(T)^c\bigr)\le \frac{C(T,\Delta t_0)(1+|x_0|_E)}{M}.
\]
\end{theo}

\begin{rem}
The condition $\|(-A)^\alpha x_0\|_H<\infty$ may be relaxed using standard arguments. If one assumes $\|(-A)^\beta x_0\|_H<\infty$ with $\beta\in[0,\alpha]$, a factor of the type $t_n^{\alpha-\beta}$ needs to be introduced. To simplify notation, we only consider the case $\beta=\alpha$ and leave the details of the general case to the interested readers.
\end{rem}

Let us state three straightforward consequences of Theorem~\ref{theo:main}, presenting other standard ways to describe the error of the numerical scheme. Proofs are postponed to Section~\ref{sec:convergence_proof2}.

\begin{cor}\label{cor:1}
The numerical scheme is mean-square convergent. Precisely, for every $T\in(0,\infty)$, and any initial condition $x_0\in E$, with $\|(-A)^\alpha x_0\|_H<\infty$, then
\[
\E\bigl[\sup_{0\le n\le N_{T,\Delta t}}\|X_n-X(n\Delta t)\|_H^2\bigr]\underset{\Delta t\to 0}\to 0.
\]
\end{cor}

\begin{cor}\label{cor:2}
The numerical scheme converges in probability with order $\alpha$, for all $\alpha<\frac14$. More precisely, for every $\alpha\in(0,\frac{1}{4}),~K\in(0,\infty)$, and any initial condition $x_0\in E$, with $\|(-A)^\alpha x_0\|_H<\infty$, then
\[
\PP\Bigl(\sup_{0\le n\le N_{T,\Delta t}}\|X_n-X(n\Delta t)\|_H\ge K\Delta t^\alpha\Bigr)\underset{\Delta t\to 0}\to 0.
\]
\end{cor}

\begin{cor}\label{cor:3}
Let $T\in(0,\infty)$, $\Delta t_0\in(0,1)$, and $\alpha\in(0,\frac{1}{4})$. For every $\epsilon\in(0,1)$, there exists $M=M_{\alpha}(\epsilon,T,\Delta t_0)\in(0,\infty)$ and $C=C_{\alpha}(\epsilon,T,\Delta t_0)\in(0,\infty)$ such that for any bounded Lipschitz continuous function $\varphi:H\to\R$, with $\|\varphi\|_\infty+{\rm Lip}(\varphi)\le 1$, and for every $\Delta t\in(0,\Delta t_0]$  and $x_0\in E$, with $\|(-A)^\alpha x_0\|_H<\infty$, then
\[
\big|\E\bigl[\varphi(X(T))\bigr]-\E\bigl[\varphi(X_{N_{T,\Delta t}})\mathds{1}_{\Omega_{M}^{(\Delta t)}(T)^c}\bigr]\big|\le \epsilon+C\Delta t^\alpha.
\]
\end{cor}

This weak convergence result is not expected to be optimal. First, it is based on the concept of rejecting exploding trajectories: with a more technical analysis, it may be possible to remove the $\epsilon$ error term. Second, as will be confirmed by the numerical experiments below, the order of convergence $\alpha$ may be replaced with $2\alpha$, using the standard weak convergence analysis, for functions $\varphi$ of class $\mathcal{C}^2$, bounded and with bounded derivatives. These improvements will be investigated in future works.

\subsection{Proof of Theorem~\ref{theo:main}}\label{sec:convergence_proof1}

\subsubsection{Two auxiliary lemmas}

We first state two auxiliary results.
\begin{lemma}\label{lem:error_noise}
Let $T\in(0,\infty)$ and $\Delta t_0\in(0,1)$. For every $\alpha<\frac{1}{4}$, there exists $C_\alpha(T,\Delta t_0)\in(0,\infty)$ such that for all $\Delta t\in(0,\Delta t_0]$
\[
\E\bigl[\underset{0\le n\le N_{T,\Delta t}}\sup \|\omega_n-W^A(n\Delta t)\|_{H}^{2}\bigr]\le C_\alpha(T,\Delta t_0)\Delta t^{2\alpha}.
\]
\end{lemma}

\begin{proof}[Proof of Lemma~\ref{lem:error_noise}]
If $\Gamma=\Gamma^{{\rm exact}}$, then $\omega_n=W^A(n\Delta t)$, and there is nothing to prove.
When $\Gamma=\Gamma^{{\rm expo}}$ or $\Gamma^{{\rm imp}}$, then
\[
\omega_n=\sum_{k=0}^{n-1}\int_{k\Delta t}^{(k+1)\Delta t}S_{\Delta t}^{n-k}dW(t).
\]
and it is known that, for all $p\in\N$, there exists $C_{p,\alpha}(T,\Delta t_0)$ such that for all $\Delta t\in(0,\Delta t_0]$,
\[
\underset{0\le n\le N_{T,\Delta t}}\sup \E\|\omega_n-W^A(n\Delta t)\|_{H}^{p}\le C_{p,\alpha}(T,\Delta t_0)\Delta t^{\alpha p}.
\]
See for instance~\cite{Printems:01} for details.

It remains to control the expectation of the supremum. Let $\alpha\in(0,\frac14)$, and, set $\tilde{\alpha}=\alpha+\frac{1}{p}$, with $p\in\N$, $p\ge 2$, chosen sufficiently large to have $\tilde{\alpha}<\frac14$. Then
\begin{align*}
\E\bigl[\underset{0\le n\le N_{T,\Delta t}}\sup\|\omega_n-W^A(n\Delta t)\|_{H}^{p}\bigr]&\le \E\bigl[\underset{0\le n\le N_{T,\Delta t}}\sum\|\omega_n-W^A(n\Delta t)\|_{H}^{p}\bigr]\\
&\le N_{T,\Delta t}C_{p,\tilde{\alpha}}(T,\Delta t_0)\Delta t^{\tilde{\alpha} p}\\
&\le TC_{p,\tilde{\alpha}}(T,\Delta t_0)\Delta t^{\tilde{\alpha}p-1},
\end{align*}
with $\tilde{\alpha}p-1=p\alpha$. Then, since $p\ge 2$
\[
\Bigl(\E\bigl[\underset{0\le n\le N_{T,\Delta t}}\sup\|\omega_n-W^A(n\Delta t)\|_{H}^{2}\Bigr)^{\frac12}\le \Bigl(\E\bigl[\underset{0\le n\le N_{T,\Delta t}}\sup\|\omega_n-W^A(n\Delta t)\|_{H}^{p}\Bigr)^{\frac{1}{p}}\le C_{\alpha}(T,\Delta t_0)\Delta t^\alpha.
\]
This concludes the proof.
\end{proof}

\begin{lemma}\label{lem:regul_time}
Let $T\in(0,\infty)$ and $\Delta t_0\in(0,1)$. For every $p\in\N$ and $\alpha<\frac{1}{4}$, there exists $C_{p,\alpha}(T)\in(0,\infty)$, such that for any initial condition $x_0$ with $\|(-A)\|^\alpha x_0\|_H<\infty$,  and all $\Delta t\in(0,\Delta t_0]$, then for all $t,s\in[0,T]$,
\[
\E\big[\|X^{(\Delta t)}(t)-X^{(\Delta t)}(s)\|_H^{2p}\bigr]\le C_{p,\alpha}(T)(1+\|(-A)\|^\alpha x_0\|_H^{2p})\Delta t^{2p\alpha}.
\]
\end{lemma}

\begin{proof}[Proof of Lemma~\ref{lem:regul_time}]
We give a sketch of proof, and give details only for $p=1$.

Then, let $0\le s<t\le T$. Then
\begin{align*}
\E\|X^{(\Delta t)}(t)-X^{(\Delta t)}(s)\|_H^2&\le C(T)\|e^{tA}x_0-e^{sA}x_0\|_H^2\\
&+C(T)\E\|W^A(t)-W^A(s)\|_H^2\\
&+C(T)\int_{s}^{t}\E\|\Psi_{\Delta t}(X^{(\Delta t)}(r)\|^2dr\\
&+C(T)\int_{0}^{s}\frac{(t-s)^{2\alpha}}{(s-r)^{2\alpha}}\E\|\Psi_{\Delta t}(X^{(\Delta t)}(r))\|_H^2dr.
\end{align*}
First, $\|e^{tA}x_0-e^{sA}x_0\|_H\le |t-s|^\alpha\|(-A)^\alpha x_0\|_H$.

The inequality $\E\|W^A(t)-W^A(s)\|_H^2\le C_\alpha(T)\le |t-s|^{2\alpha}$ is a standard result.

Finally, it remains to apply Lemma~\ref{lem:Lip_Psi}, and the moment bound from Proposition~\ref{propo:aux} to conclude.
\end{proof}

\subsubsection{Error between solutions of exact and auxiliary equations}
The following result states convergence of the $X^{(\Delta t)}$ to $X$ when $\Delta t$ goes to $0$. The order of convergence is $1$, and there is no need for localization.
\begin{propo}\label{propo:modif_error}
Let $T\in(0,\infty)$ and $\Delta t_0\in(0,1)$. There exists $C(T,\Delta t_0)\in(0,\infty)$ such that for all $x_0\in E$ and $\Delta t\in(0,\Delta t_0]$,
\[
\E\Bigl[\underset{0\le t\le T}\sup\|X^{(\Delta t)}(t)-X^{(0)}(t)\|_H^2\Bigr]\le C(T,\Delta t_0)\bigl(1+|x_0|_{E}^{10}\bigr)\Delta t^2.
\]
\end{propo}

\begin{proof}[Proof of Proposition~\ref{propo:modif_error}]
Let $R^{(\Delta t)}(t)=X^{(\Delta t)}(t)-X^{(0)}(t)$. Then
\[
dR^{(\Delta t)}(t)=AR^{(\Delta t)}(t)dt+\bigl(\Psi_{\Delta t}(X^{(\Delta t)}(t))-\Psi_{0}(X^{(0)}(t))\bigr)dt,
\]
with $R^{(\Delta t)}(0)=0$. As a consequence,
\begin{align*}
\frac{1}{2}\frac{d\|R^{(\Delta t)}(t)\|^2}{dt}&=\langle AR^{(\Delta t)}(t),R^{(\Delta t)}(t)\rangle+\langle \Psi_{\Delta t}(X^{(\Delta t)}(t))-\Psi_{0}(X^{(0)}(t)),R^{(\Delta t)}(t)\rangle\\
&\le \langle \Psi_{\Delta t}(X^{(\Delta t)}(t))-\Psi_{\Delta t}(X^{(0)}(t)),R^{(\Delta t)}(t)\rangle\\
&~+\|\Psi_{\Delta t}(X^{(0)}(t))-\Psi_{0}(X^{(0)}(t))\|\|R^{(\Delta t)}(t)\|
\end{align*}

Using the one-sided Lipschitz condition from Lemma~\ref{lem:one-sided}, and Lemma~\ref{lem:errorPsi}, then
\begin{align*}
\frac{1}{2}\frac{d\|R^{(\Delta t)}(t)\|^2}{dt}&\le \bigl(e^{\Delta t_0}+\frac{1}{2}\bigr)\|R^{(\Delta t)}(t)\|^2+\frac{1}{2}\|\Psi_{\Delta t}(X^{(0)}(t))-\Psi_{0}(X^{(0)}(t))\|^2\\
&\le C\|R^{(\Delta t)}(t)\|^2+C\Delta t^2\bigl(1+|X^{(0)}(t)|_E^{10}\bigr).
\end{align*}

Using Gronwall's lemma, and Proposition~\ref{propo:AC} then concludes the proof.
\end{proof}

\subsubsection{Proof of Theorem~\ref{theo:main}}

We are now in position to prove the main result of this article,  Theorem~\ref{theo:main}. Thanks to Proposition~\ref{propo:modif_error}, it is sufficient to look at the error $X_n-X^{(\Delta t)}(n\Delta t)$.

Define $r_n=X_n-\omega_n$. Note that
\[
X_{n+1}-S_{\Delta t}X_n-\Delta tS_{\Delta t}\Psi_{\Delta t}(X_n)=\omega_{n+1}-S_{\Delta t}\omega_n,
\]
thus
\[
r_{n+1}=S_{\Delta t}r_n+\Delta tS_{\Delta t}\Psi_{\Delta t}(X_n).
\]
This identity yields (since $r_0=X_0=x_0$)
\[
X_n-\omega_n=r_n=S_{\Delta t}^nx_0+\Delta t\sum_{k=0}^{n-1}S_{\Delta t}^{n-k}\Psi_{\Delta t}(X_k).
\]

We are now in position to make precise the decomposition of the error:
\begin{align*}
X_n-X^{(\Delta t)}(n\Delta t)&=\bigl(S_{\Delta t}^n-e^{n\Delta t}\bigr)x_0+\omega_n-W^A(n\Delta t)\\
&+\sum_{k=0}^{n-1}\int_{k\Delta t}^{(k+1)\Delta t}\bigl[S_{\Delta t}^{n-k}\Psi_{\Delta t}(X_k)-e^{(n\Delta t-t)A}\Psi_{\Delta t}(X^{(\Delta t)}(t))\bigr]dt.
\end{align*}

First, there exists $C_\alpha\in(0,\infty)$, such that for all $n\in\N$,
\[
\|\bigl(S_{\Delta t}^n-e^{n\Delta t}\bigr)x_0\|_H\le C_\alpha \Delta t^\alpha\|(-A)^\alpha x_0\|_H.
\]
In addition, the error term $\omega_n-W^A(n\Delta t)$ is controlled thanks to Lemma~\ref{lem:error_noise}.

It remains to deal with
\begin{align*}
\sum_{k=0}^{n-1}\int_{k\Delta t}^{(k+1)\Delta t}&\bigl[S_{\Delta t}^{n-k}\Psi_{\Delta t}(X_k)-e^{(n\Delta t-t)A}\Psi_{\Delta t}(X^{(\Delta t)}(t))\bigr]dt\\
&=\Delta t\sum_{k=0}^{n-1}S_{\Delta t}^{n-k}\bigl[\Psi_{\Delta t}(X_k)-\Psi_{\Delta t}(X^{(\Delta t)}(k\Delta t))\bigr]\\
&+\sum_{k=0}^{n-1}\int_{k\Delta t}^{(k+1)\Delta t}S_{\Delta t}^{n-k}\bigl[\Psi_{\Delta t}(X^{(\Delta t)}(k\Delta t)-\Psi_{\Delta t}(X^{(\Delta t)}(t))\bigr]dt\\
&+\sum_{k=0}^{n-1}\int_{k\Delta t}^{(k+1)\Delta t}\bigl[S_{\Delta t}^{n-k}-e^{(n\Delta t-t)A}\bigr]\Psi_{\Delta t}(X^{(\Delta t)}(t))dt
\end{align*}

Since $\Psi_{\Delta t}$ is not globally Lipschitz continuous uniformly in $\Delta t\in(0,\Delta t_0)$, a localization argument is introduced.

For $M\in\N$, and $n\in\left\{0,\ldots,N_{T,\Delta t}\right\}$, let
\[
\Omega_{n,M}^{(\Delta t)}=\left\{\underset{0\le k\le n}\sup|X_k|_E+\underset{0\le t\le n\Delta t}\sup|X^{(\Delta t)}(t)|_E\le M\right\}.
\]
Note that for all $k\in\left\{0,\ldots,n-1\right\}$, $0\le \mathds{1}_{\Omega_{n,M}^{(\Delta t)}}\le \mathds{1}_{\Omega_{k,M}^{(\Delta t)}}\le 1$.

Let $\epsilon_n=\mathds{1}_{\Omega_{n,M}^{(\Delta t)}}\|X_n-X^{(\Delta t)}(n\Delta t)\|_H^2$. Then
\begin{align*}
\E\bigl[\underset{0\le m\le n}\sup\epsilon_m\bigr]&\le C\sup_{0\le m\le n}\|\bigl(S_{\Delta t}^m-e^{m\Delta t}\bigr)x_0\|_H^2+C\E\bigl[\sup_{0\le m\le n}\|\omega_m-W^A(m\Delta t)\|_H^2\bigr]\\
&+CT\Delta t \sum_{k=0}^{n-1}\E\bigl[\mathds{1}_{\Omega_{k,M}^{(\Delta t)}}\|\Psi_{\Delta t}(X_k)-\Psi_{\Delta t}(X^{(\Delta t)}(k\Delta t))\|^2\bigr]\\
&+CT\sum_{k=0}^{n-1}\int_{k\Delta t}^{(k+1)\Delta t}\E\|\Psi_{\Delta t}(X^{(\Delta t)}(k\Delta t)-\Psi_{\Delta t}(X^{(\Delta t)}(t))\|^2dt\\
&+CT\sum_{k=0}^{n-1}\int_{k\Delta t}^{(k+1)\Delta t}\E\big\|\bigl[S_{\Delta t}^{n-k}-e^{(n\Delta t-t)A}\bigr]\Psi_{\Delta t}(X^{(\Delta t)}(t))\big\|_H^2dt\\
&\le C_{\alpha}(T)(1+\|(-A)^\alpha x_0\|_H^2)\Delta t^{2\alpha}\\
&+C(1+M^6)T\Delta t \sum_{k=0}^{n-1}\E[\underset{0\le m\le k}\sup\epsilon_m]\\
&+C_\alpha(T)(1+M^6)\Delta t^{2\alpha}\\
&+C_\alpha(T)(1+|x_0|_E^8)\Delta t^{2\alpha}.
\end{align*}
We have first used the estimate above and Lemma~\ref{lem:error_noise}. Moreover, if $x_1,x_2\in E$ satisfy $|x_1|_E\le M,~|x_2|_E\le M$, then
\[
\|\Psi_{\Delta t}(x_2)-\Psi_{\Delta t}(x_1)\|_H\le C(\Delta t_0)(1+M^3)\|x_2-x_1\|_H.
\]
Thus
\begin{align*}
\E\bigl[\mathds{1}_{\Omega_{k,M}^{(\Delta t)}}\|\Psi_{\Delta t}(X_k)-\Psi_{\Delta t}(X^{(\Delta t)}(k\Delta t))\|^2\bigr]&\le C(1+M^6) \E\bigl[\mathds{1}_{\Omega_{k,M}^{(\Delta t)}}\|X_k-X^{(\Delta t)}(k\Delta t)\|^2\bigr]\\
&\le C(1+M^6)\E[\underset{0\le m\le k}\sup\epsilon_m].
\end{align*}

Finally, we have used Lemma~\ref{lem:regul_time}, and the standard estimates to control $\|S_{\Delta t}^{n-k}-e^{(n\Delta t-t)A}\|_{\mathcal{L}(H)}$.

Applying the discrete Gronwall's Lemma,
\[
\E\bigl[\underset{0\le m\le n}\sup\epsilon_m\bigr]\le C_\alpha(T,M)\Delta t^{2\alpha}(1+|x_0|_E^8+\|(-A)^\alpha x_0\|_H^2),
\]
with $K_\alpha(T,M)\le C_\alpha(T)(1+M^6)\exp(C_\alpha(T)M^6 T)$.

To conclude, note that $\Omega_{N_{T,\Delta t},M}^{(\Delta t)}=\Omega_{M}^{(\Delta t)}(T)$, and that
\[
\E\bigl[\mathds{1}_{\Omega_{N_{T,\Delta t},M}^{(\Delta t)}}\sup_{0\le n\le N_{T,\Delta t}}\|X_n-X^{(\Delta t)}(n\Delta t)\|_H^2]\le \E[\underset{0\le n\le N_{T,\Delta t}}\sup\epsilon_n].
\]

In addition,
\[
1-\PP\bigl(\Omega_{N_{T,\Delta t},M}^{(\Delta t)}\bigr)\le \frac{\E\bigl[\underset{0\le k\le N_{T,\Delta t}}\sup|X_k|_E+\underset{0\le t\le T}\sup|X^{(\Delta t)}(t)|_E\bigr]}{M}\le \frac{C(T,\Delta t_0)(1+|x_0|_E)}{M},
\]
thanks to Propositions~\ref{propo:aux} and~\ref{propo:moment}.

These estimates, combined with Proposition~\ref{propo:modif_error}, conclude the proof of Theorem~\ref{theo:main}.

\subsection{Proof of the Corollaries}\label{sec:convergence_proof2}

\begin{proof}[Proof of Corollary~\ref{cor:1}]
Let $M$ be an arbitrary integer. Then, for all $\Delta t\in(0,\Delta t_0]$,
\begin{align*}
\E\bigl[\sup_{0\le n\le N_{T,\Delta t}}\|X_n-X(n\Delta t)\|_H^2]&=
\E\bigl[\mathds{1}_{\Omega_{M}^{(\Delta t)}(T)}\sup_{0\le n\le N_{T,\Delta t}}\|X_n-X(n\Delta t)\|_H^2]\\
&~+\E\bigl[\mathds{1}_{\Omega_{M}^{(\Delta t)}(T)^c}\sup_{0\le n\le N_{T,\Delta t}}\|X_n-X^{(\Delta t)}(n\Delta t)\|_H^2]\\
&\le C_\alpha(M,T,\Delta t_0,x)\Delta t^{2\alpha}\\
&~+\PP\bigl(\Omega_{M}^{(\Delta t)}(T)^c\bigr)^{\frac12}\bigl(\E\bigl[\sup_{0\le n\le N_{T,\Delta t}}\|X_n\|_H^4+\|X(n\Delta t)\|_H^4]\bigr)^{\frac12}\\
&\le C_\alpha(M,T,\Delta t_0,x)\Delta t^{2\alpha}+\frac{C_\alpha(T,\Delta t_0,x)}{\sqrt{M}},
\end{align*}
thanks to Theorem~\ref{theo:main}, Cauchy-Schwarz inequality, and Propositions~\ref{propo:AC} and~\ref{propo:moment}. Thus, letting first $\Delta t\to 0$, then $M\to \infty$, yields
\[
\underset{\Delta t\to 0}\limsup~\E\bigl[\sup_{0\le n\le N_{T,\Delta t}}\|X_n-X(n\Delta t)\|_H^2]\le \frac{C_\alpha(T,\Delta t_0,x)}{\sqrt{M}}\underset{M\to\infty}\to 0,
\]
which concludes the proof of Corollary~\ref{cor:1}.
\end{proof}

\begin{proof}[Proof of Corollary~\ref{cor:2}]
Let $K\in(0,\infty)$ be an arbitrary positive real number, and $M\in\N$ be an arbitrary integer. 
Let $\tilde{\alpha}\in(\alpha,\frac{1}{4})$. Then
\begin{align*}
\PP\Bigl(\sup_{0\le n\le N_{T,\Delta t}}&\|X_n-X(n\Delta t)\|_H\ge K\Delta t^\alpha\Bigr)\le \PP\Bigl(\Omega_{M}^{(\Delta t)}(T)^c\Bigr)\\ &+\PP\Bigl(\left\{\sup_{0\le n\le N_{T,\Delta t}}\|X_n-X(n\Delta t)\|_H\ge K\Delta t^\alpha\right\} \cap \Omega_{M}^{(\Delta t)}(T) \Bigr)\\
&\le \frac{C_\alpha(T,\Delta t_0,x)}{M}+\frac{1}{K^2\Delta t^{2\alpha}}\E\bigl[\mathds{1}_{(\Omega_{M}^{(\Delta t)}(T)}\sup_{0\le n\le N_{T,\Delta t}}\|X_n-X(n\Delta t)\|_H^2\bigr]\\
&\le \frac{C_\alpha(T,\Delta t_0,x)}{M}+\frac{1}{K^2}K_{\tilde{\alpha}}(M,T,\Delta t_0)\Delta t^{2(\tilde{\alpha}-\alpha)},
\end{align*}
thanks to Theorem~\ref{theo:main}. Thus, letting first $\Delta t\to 0$, then $M\to \infty$, yields
\[
\underset{\Delta t\to 0}\limsup~\PP\Bigl(\sup_{0\le n\le N_{T,\Delta t}}\|X_n-X(n\Delta t)\|_H\ge K\Delta t^\alpha\Bigr)\le \frac{C_\alpha(T,\Delta t_0,x)}{M}\underset{M\to\infty}\to 0,
\]
which concludes the proof of Corollary~\ref{cor:2}.
\end{proof}

\begin{proof}[Proof of Corollary~\ref{cor:3}]
Let $\epsilon\in(0,1)$, and $M\in\N$ be such that, for all $\Delta t\in(0,\Delta t_0]$,
\[
\PP\bigl(\Omega_{M}^{(\Delta t)}(T)^c\bigr)\le \frac{C(T,\Delta t_0)(1+|x_0|_E)}{M}\le \epsilon.
\]
Then
\begin{align*}
\big|\E\bigl[\varphi(X(T))\bigr]-\E\bigl[\varphi(X_{N_{T,\Delta t}})\mathds{1}_{\Omega_{M}^{(\Delta t)}(T)}\bigr]\big|&\le \big|\E\bigl[\bigl(\varphi(X(T))-\varphi(X_{N_{T,\Delta t}})\bigr)\mathds{1}_{\Omega_{M}^{(\Delta t)}(T)}\bigr]\big|\\
&+\E\bigl[\varphi(X(T))\mathds{1}_{\Omega_{M}^{(\Delta t)}(T)^c}\bigr]\\
&\le \E\bigl[\big\|X(T)-X_{N_{T,\Delta t}}\big\|_H\mathds{1}_{\Omega_{M}^{(\Delta t)}(T)}\bigr]+\epsilon\\
&\le C_\alpha(T,\Delta t_0,x)\Delta t^\alpha+\epsilon,
\end{align*}
thanks to the assumption that $\varphi$ is bounded and Lipschitz continuous, with $\|\varphi\|_\infty+{\rm Lip}(\varphi)\le 1$, and to Theorem~\ref{theo:main}. This concludes the proof of Corollary~\ref{cor:3}.
\end{proof}

\section{Numerical experiments}\label{sec:num-expe}

This section is devoted to numerical simulations, in order to investigate the properties of the numerical scheme~\eqref{eq:scheme_Gamma}, with the choice $\Gamma=\Gamma^{\rm imp}$ of the linear implicit Euler scheme:
\begin{equation}\label{eq:scheme_expe}
\begin{cases}
Y_{n}=\Phi_{\Delta t}(X_n),\\
X_{n+1}=S_{\Delta t}Y_n+S_{\Delta t}\bigl(W((n+1)\Delta t)-W(n\Delta t)\bigr),
\end{cases}
\end{equation}
with $S_{\Delta t}=(I-\Delta tA)^{-1}$. All simulations are performed with this choice of integrator. Indeed, we expect that there is no gain in the orders of convergence when using the version $\Gamma=\Gamma^{\rm exp}$, with $S_{\Delta t}=e^{\Delta t A}$. In addition, computing such exponential operators may be expensive in more complex situations, for instance where eigenvalues and eigenfunctions of $A$ are not explicitly known, or in higher dimensional domains. It is thus natural to restrict our simulations to the linear implicit Euler scheme.

Spatial discretization is performed using a standard finite differences scheme, with a fixed mesh size. The dependence of the error with respect to this spatial discretization parameter is not studied in this article: we only focus on the temporal discretization error.

Variants of the scheme~\eqref{eq:scheme_expe} are introduced below, in Section~\ref{sec:num-expe_schemes}. They are based on other splitting strategies. The numerical simulations allow us to compare the orders of convergence of these methods.

First, in Section~\ref{sec:num-expe_strong}, strong orders of convergence of the schemes are compared. We observe that in practice the result of Theorem~\ref{theo:main} holds true without requiring the introduction of the set $\Omega_M^{(\Delta t)}(T)$, and that all the methods are expected to have the same order of convergence, equal to $1/4$. We conjecture that the strong order of convergence of the scheme~\eqref{eq:scheme_expe} is equal to $1/4$.

Second, in Section~\ref{sec:num-expe_weak}, weak orders of convergence of the schemes are compared. Note that the rejection of exploding trajectories, as suggested by Corollary~\ref{cor:3}, is not performed: we may take $\epsilon=0$. Moreover, the test function is of class $\mathcal{C}^2$, bounded and with bounded derivatives. In addition, one of the alternative splitting schemes defined below has a lower weak error. Based on these numerical simulations, we conjecture that the weak order of convergence is then equal to $1/2$ for the scheme~\eqref{eq:scheme_expe}. This question will be studied in future works.

We also plan to study generalizations in higher dimension.

\subsection{Variants of the numerical scheme~\eqref{eq:scheme_expe}}\label{sec:num-expe_schemes}

We define three numerical schemes, for each value of the time-step size $\Delta t>0$. We recall that $S_{\Delta t}=(I-\Delta tA)^{-1}$, and use the notation $\Delta W_n=W((n+1)\Delta t)-W(n\Delta t)$ for Wiener increments.

Method~1, given by the scheme~\eqref{eq:M1}, corresponds with the scheme studied above,~\eqref{eq:scheme_Gamma}, with the linear implicit Euler integrator. The definition of Method~3, given by the scheme~\eqref{eq:M3}, is motivated by~\cite{KovacsLarssonLindgren:15_2}. The numerical experiments below show that the error is reduced when using this scheme, but the order of convergence seems to be the same. Finally, the definition of Method~3 is motivated by~\cite{BrehierGazeauGoudenegeRousset:15}. We have checked that the three variants give consistent results. In addition, the observations are stable with respect to the choice of the mesh size.

\subsubsection*{Method~1}

\begin{equation}\label{eq:M1}
\begin{cases}
Y_{n}^1=\Phi_{\Delta t}(X_n^1),\\
X_{n+1}^1=S_{\Delta t}X_n^1+S_{\Delta t}\Delta W_n,
\end{cases}
\end{equation}

\subsubsection*{Method~2}

\begin{equation}\label{eq:M3}
\begin{cases}
Y_n^{2,1}=S_{\frac{\Delta t}{2}}X_n^{2}\\
Y_n^{2,2}=\Phi_{\Delta t}(Y_n^{2,1})\\
X_{n+1}^2=S_{\frac{\Delta t}{2}}\bigl(Y_n^{2,2}+\Delta W_n\bigr).
\end{cases}
\end{equation}

\subsubsection*{Method~3}

\begin{equation}\label{eq:M6}
\begin{cases}
Y_n^{3,1}=S_{\frac{\Delta t}{2}}\bigl(X_n^3+\frac{1}{2}\Delta W_n\bigr)\\
Y_n^{3,2}=\Phi_{\Delta t}(Y_n^{3,1})\\
X_{n+1}^3=S_{\frac{\Delta t}{2}}\bigl(Y_n^{3,2}+\frac{1}{2}\Delta W_n\bigr).
\end{cases}
\end{equation}

\begin{rem}
Using different splitting strategies yields other numerical schemes. For instance, we may have considered the scheme defined by
\[
\begin{cases}
Y_n^1=\Phi_{\frac{\Delta t}{2}}(X_n)\\
Y_n^2=S_{\Delta t}\bigl(Y_n^1+\Delta W_n)\\
X_{n+1}=\Phi_{\frac{\Delta t}{2}}(Y_n^2).
\end{cases}
\]
Numerical experiments for this scheme are not reported, since they do not differ from Method~1.
\end{rem}

\subsection{Strong convergence}\label{sec:num-expe_strong}

In order to study the strong order of convergence, one needs to compare trajectories computed using the same Wiener path, which constraints the construction of the associated Wiener increments. It is customary to compare the numerical solution computed with time-step size $\Delta t$, with a reference solution computed using a much smaller time-step size. Instead, we estimate the mean-square error
\[
\E\|X_{N}^{(\Delta t)}-X_{2N}^{(\frac{\Delta t}{2})}\|^2
\]
where $X^{(\Delta t)}_{N}$ is the numerical scheme, with time-step size $\Delta t$, and $N\Delta t=T$. The solutions are computed using the same Wiener path for one value of $\Delta t$, and using independent Wiener paths when changing $\Delta t$. One needs to check that this error is bounded from above by $C_\alpha(T)\Delta t^{2\alpha}$: by a telescoping sum argument, since by Corollary~\ref{cor:2} the scheme is mean-square convergent, this property is equivalent to a standard error estimate
\[
\E\|X_{N}^{(\Delta t)}-X(T)\|^2\le C_\alpha '(T)\Delta t^{2\alpha}.
\]
In addition, note that the order of convergence of the error $\E\|X_{N}^{(\Delta t)}-X_{2N}^{(\frac{\Delta t}{2})}\|^2$ is exactly what matters in the analysis of Multilevel Monte Carlo algorithms, which are used in the context of weak convergence -- see Section~\ref{sec:num-expe_weak} below.

The simulations are performed with $T=1$, a mesh size $\Delta x=2.5~10^{-4}$, and computing Monte-Carlo averages over $10^5$ independent realizations. The numerical results, in logarithmic scale, are reported in Figure~\ref{fig:strong}. We observe that the mean-square error converges with order $2\alpha=1/2$, for the three methods.

\begin{figure}
\includegraphics[scale=0.8]{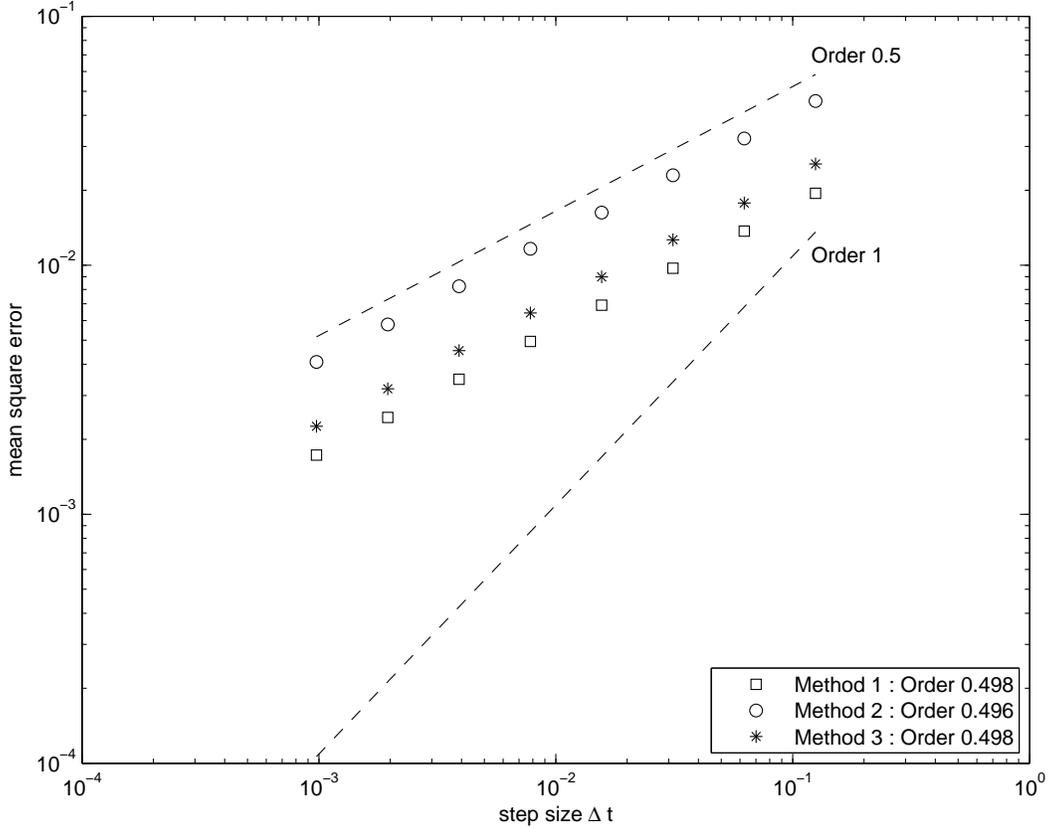}
\caption{Mean square error order for $T=1$, $\Delta x=2.5~10^{-4}$ and $10^{5}$ independent realizations.}
\label{fig:strong}
\end{figure}

\subsection{Weak convergence}\label{sec:num-expe_weak}

Weak orders of convergence deal to the behavior of the error
\[
\E\left[\varphi\left(X_N^{(\Delta t)}\right)\right]-\E\bigl[\varphi(X(T))\bigr]
\]
where $\varphi:H\to \R$ is a test function, with appropriate regularity properties. Precisely, in the numerical experiments below, the test function is given by
\[
\varphi(x)=\exp\bigl(-5\|x\|_H^2\bigr)
\]
which is of class $\mathcal{C}^2$, bounded and with bounded derivatives. Our aim is to check that the weak order of convergence is equal to $2\alpha=1/2$, where $\alpha=1/4$ is the strong order.

Experiments to identify weak rates of convergence are plagued by statistical error, and thus we need to use a variance reduction strategy. Instead of directly comparing $\E\bigl[\varphi(X_N^{(\Delta t)})\bigr]$ with a reference value, estimated by an independent Monte Carlo experiment with much smaller time step, we use a form of Multilevel Monte Carlo method. Precisely, we estimate (by a standard Monte Carlo average) the error
\[
\E\left[\varphi\left(X_N^{(\Delta t)}\right)\right]-\E\left[\varphi\left(X_{2N}^{(\frac{\Delta t}{2})}\right)\right]
\]
using the same Wiener paths (as explained in Section~\ref{sec:num-expe_strong}), but different time-step sizes, respectively $\Delta t$ and $\frac{\Delta t}{2}$. Between two successive levels, the time-step size $\Delta t$ is decreased, and computations at different levels use independent Wiener paths. Contrary to the standard Multilevel Monte Carlo strategy, the number of realizations per level is not optimized (it is the same at each level): still the computational cost is significantly reduced (thanks to the strong convergence property checked in Section~\ref{sec:num-expe_strong}), and the observation of the weak orders of convergence is improved a lot.

The comparison of $\E\left[\varphi\left(X_N^{(\Delta t)}\right)\right]$ with a reference value $\E\left[\varphi\left(X_N^{(\frac{\Delta t}{2^K})}\right)\right]$ estimated with a smaller time step $\frac{\Delta t}{2^K}$ is performed using a straightforward telescoping sum procedure.

The simulations are performed with $T=1$, $\Delta x=2.5~10^{-4}$, and with $10^5$ independent Monte Carlo realizations at each level. The numerical results, in logarithmic scale, are reported in Figure~\ref{fig:weak}. We observe that the weak error converges with order $2\alpha=1/2$ for three methods, and that Method~3 seems more efficient.

\begin{figure}
\includegraphics[scale=0.8]{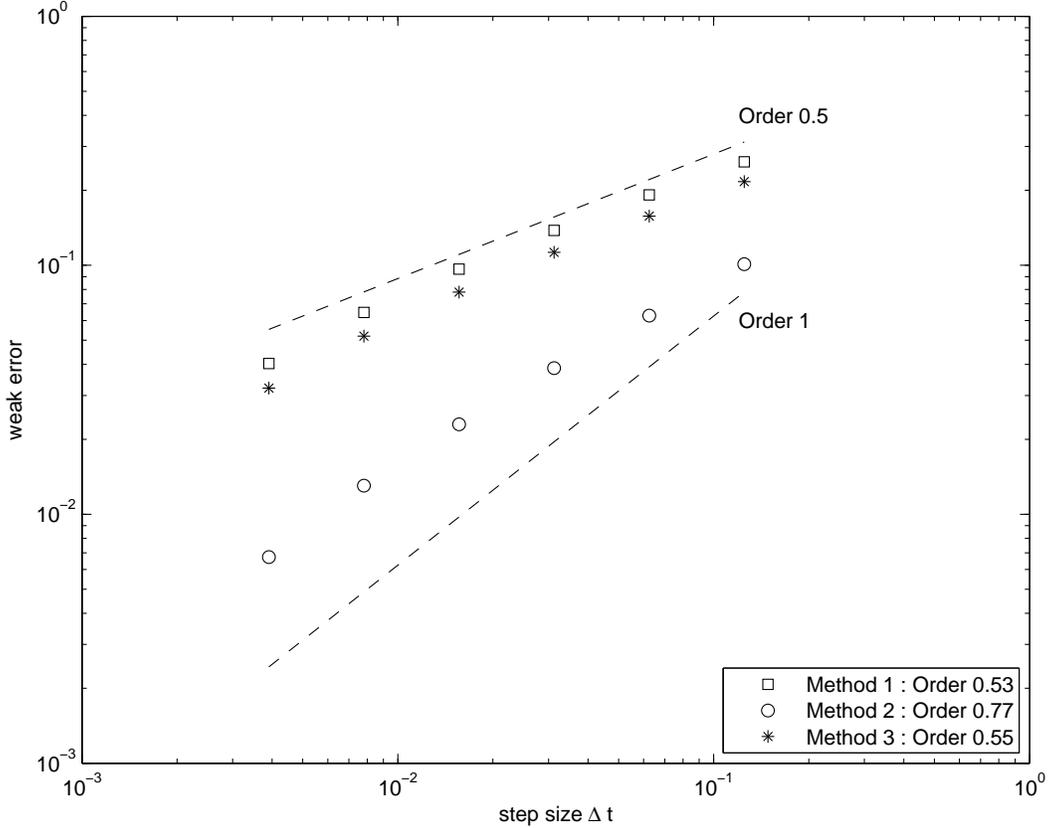}
\caption{Weak error order for $T=1$, $\Delta x=2.5~10^{-4}$ and $10^{5}$ independent realizations.}
\label{fig:weak}
\end{figure}

%
%
%
%
%
%
%
%

\bibliographystyle{abbrv}
\bibliography{refAC}

\begin{appendix}
\section{Proof of auxiliary results}\label{appendix}

Let $\Delta t_0\in(0,1)$, and $\Delta t\in(0,1)$. Note that the properties are straightforward when $\Delta t=0$: then $\Phi_0(z)=z$ and $\Psi_0(z)=z-z^3$.

Recall (see~\eqref{eq:defPhi} and~\eqref{eq:defPsi}) that, for all $\Delta t\ge 0$ and $z\in\R$,
\[
\Phi_{\Delta t}(z)=\frac{z}{\sqrt{z^2+(1-z^2)e^{-2\Delta t}}}=\frac{z}{\sqrt{e^{-2\Delta t}+(1-e^{-2\Delta t})z^2}}~,\quad \Psi_{\Delta t}(z)=\frac{\Phi_{\Delta t}(z)-z}{\Delta t}.
\]

\begin{proof}[Proof of Lemma~\ref{lem:Lip_Phi}]
The mapping $\Phi_{\Delta t}$ is of class $\mathcal{C}^1$, and for all $z\in\R$,
\[
\frac{d}{dz}\Phi_{\Delta t}(z)=\frac{e^{-2\Delta t}}{\bigl(e^{-2\Delta t}+(1-e^{-2\Delta t})z^2\bigr)^{3/2}}\in [0,e^{\Delta t}].
\]
The conclusion is straightforward.
\end{proof}

\begin{proof}[Proof of Lemma~\ref{lem:one-sided}]
We claim that, for all $z\in\R$,
\[
\frac{d\Psi_{\Delta t}(z)}{dz}\le e^{\Delta t}.
\]
To get this estimate, first compute
\begin{align*}
\frac{d\Psi_{\Delta t}(z)}{dz}&=\frac{d}{dz}\Bigl(\frac{z}{\Delta t}\bigl(\frac{1}{\sqrt{e^{-2\Delta t}+(1-e^{-2\Delta t})z^2}}-1\bigr)\Bigr)\\
&=\frac{1}{\Delta t}\bigl(\frac{e^{-2\Delta t}}{\bigl[e^{-2\Delta t}+(1-e^{-2\Delta t})z^2\bigr]^{3/2}}-1\bigr)\\
&=\frac{1}{\Delta t}\bigl(f_z(\Delta t)-f_z(0)\bigr),
\end{align*}
where, for all $t\ge 0$ and $z\in\R$,
\[
f_z(t)=\frac{e^{-2t}}{\bigl[z^2+(1-z^2)e^{-2t}\bigr]^{3/2}}.
\]

Then, for fixed $z\in\R$, and all $t\ge 0$, compute
\begin{align*}
f_z'(t)&=f_z(t)\Bigl(-2+\frac{3(1-z^2)e^{-2t}}{z^2+(1-z^2)e^{-2t}}\Bigr)\\
&=f_z(t)-\frac{3z^2f_z(t)}{z^2+(1-z^2)e^{-2t}}\\
&\le f_z(t).
\end{align*}
Thanks to Gronwall's Lemma, and $f_z(0)=1$, one gets for all $t\ge 0$, and fixed $z\in\R$,
\[
f_z'(t)\le f_z(t)\le e^t.
\]
Then, for all $z\in\R$,
\[
\frac{d\Psi_{\Delta t}(z)}{dz}=\frac{f_z(\Delta t)-f_z(0)}{\Delta t}\le e^{\Delta t}
\]
which concludes the proof of the claim. Concluding the proof of Lemma~\ref{lem:one-sided} is then straightforward.
\end{proof}

\begin{proof}[Proof of Lemma~\ref{lem:Lip_Psi}]
Following the computations from the proof of Lemma~\ref{lem:one-sided} above, for all $z\in\R$,
\[
\left|\frac{d\Psi_{\Delta t}(z)}{dz}\right|=\frac{1}{\Delta t}\big|f_z(\Delta t)-f_z(0)\big|\le 3e^{3\Delta t}(1+|z|^2).
\]
Indeed, for fixed $z\in\R$, and $t\ge 0$, using that $f_z(t)\in[0,e^{t}]$,
\[
\big|f_z'(t)\big|\le f_z(t)+\frac{3|z|^2f_z(t)}{e^{-2t}+(1-e^{-2t})|z|^2}\le e^{t}+3|z|^2e^{3t}.
\]
This concludes the proof of the first estimate. The second estimate is straightforward, since $\Psi_{\Delta t}(0)=0$.
\end{proof}

\begin{proof}[Proof of Lemma~\ref{lem:errorPsi}]
Let $z\in\R$ be fixed. Note that
\[
\Psi_{\Delta t}(z)-\Psi_0(z)=\frac{z}{\Delta t}\bigl(g_z(\Delta t)-g_z(0)-\Delta tg_z'(0)\bigr)
\]
with $g_z(t)=\frac{1}{\sqrt{z^2+(1-z^2)e^{-2t}}}$, and $g_z'(0)=1-z^2$.

The derivatives of $g_z$ satisfy, for all $t\ge 0$,
\begin{align*}
g_z'(t)&=g_z(t)\Bigl(1-\frac{z^2}{z^2+(1-z^2)e^{-2t}}\Bigr)\\
g_z''(t)&=g_z(t)\Bigl[\Bigl(1-\frac{z^2}{z^2+(1-z^2)e^{-2t}}\Bigr)^2-\frac{z^2(1-z^2)e^{-2t}}{(z^2+(1-z^2)e^{-2t})^2}\Bigr].
\end{align*}

Note that $z^2+(1-z^2)e^{-2t}=e^{-2t}+(1-e^{-2t})z^2\ge e^{-2t}\ge e^{-2\Delta t_0}$, when $0\le t\le \Delta t \le \Delta t_0$. Then it is straightforward to check that for all $t\in[0,\Delta t]$
\[
|g_z(t)|\le e^{\Delta t_0}~,\quad |g_z'(t)|\le e^{\Delta t_0}(1+e^{2\Delta t_0}|z|^2)~,\quad |g_z''(t)|\le 2e^{\Delta t_0}\bigl(1+e^{2\Delta t_0}|z|^2\bigr)^2.
\]
Thus $|g_z''(t)|\le C(\Delta t_0)(1+|z|^4)$, and applying Taylor's formula then concludes the proof.
\end{proof}

\end{appendix}

\end{document}